\documentclass[11pt]{amsart}
\usepackage[english,activeacute]{babel}
\usepackage{amsmath,amsfonts,amssymb,amstext,amsthm,amscd,mathrsfs,amsbsy}
\usepackage{graphics,graphicx}

\newtheorem{teo}{Theorem}[section]
\newtheorem{defi}[teo]{Definition}

\newtheorem{pro}[teo]{Proposition}

\newtheorem{lem}[teo]{Lemma}
\newtheorem{rem}[teo]{Remark}
\newtheorem{exam}[teo]{Example}
\newcommand{\R}{\mathbb R}
\newcommand{\N}{\mathbb N}
\newcommand{\Z}{\mathbb Z}
\newcommand{\Q}{\mathbb Q}
\newcommand{\C}{\mathbb C}

\newcommand{\Si}{\mbox{Sing}}
\newcommand{\Co}{\mbox{Conv}}
\newcommand{\F}{\mathbb C(x_1,x_2)}

\begin{document}

\title{Nash modification on toric surfaces}
\author{Daniel Duarte}

\address{Universit\'{e} Paul Sabatier, Institut de Math\'{e}matiques de Toulouse, 118 route de Narbonne, F-31062
Toulouse Cedex 9}
\email{dduarte@math.univ-toulouse.fr}
\thanks{Research supported by CONACYT (M\'{e}xico)}
\date{}
\dedicatory{Dedicated to Heisuke Hironaka on the occasion of his 80th birthday}
\maketitle

\begin{abstract}
It has been recently shown that the iteration of Nash modification on not necessarily normal toric
varieties corresponds to a purely combinatorial algorithm on the generators of the semigroup associated
to the toric variety. We will show that for toric surfaces this algorithm stops for certain choices of
affine charts of the Nash modification. In addition, we give a bound on the number of steps required
for the algorithm to stop in the cases we consider. Let $\C(x_1,x_2)$ be the field of rational functions of a
toric surface. Then our result implies that if $\nu:\C(x_1,x_2)\rightarrow\Gamma$ is any valuation centered on
the toric surface and such that $\nu(x_1)\neq\lambda\nu(x_2)$ for all $\lambda\in\R\setminus\Q$, then
a finite iteration of Nash modification gives local uniformization along $\nu$.
\end{abstract}




\section*{\large Introduction}

We are interested in applying the Nash modification to not necessarily normal toric surfaces and finding out
whether or not the iteration of this process resolves their singularities. The Nash modification
of an equidimensional algebraic variety replaces singular points by limits of tangent spaces to non-singular
points. Following the work of Nobile (\cite{N}), Rebassoo showed in his thesis (\cite{R}) that the iteration of
Nash modification resolves the singularities of the family $\{z^p+x^qy^r=0\}\subset\C^3$, for any positive integers
$p,$ $q$, $r$ without a common divisor. In the context of normal toric varieties (over an algebraically closed field
of characteristic zero), Gonzalez-Sprinberg (\cite{GS-1}), and later Lejeune-Jalabert and Reguera (\cite{LJ-R}), have
exhibited an ideal, called the log-jacobian ideal, whose blowing-up is the Nash modification of the toric variety.
Using the work of Gonzalez-Sprinberg (\cite{GS-1}, \cite{GS-2}), and Hironaka (\cite{H}), Spivakovsky (\cite{Sp})
proved that iterating Nash modification composed with normalization resolves singularities of surfaces. More recently,
normalized Nash modification has appeared in the work of Atanasov et al. (\cite{At}). Moreover, it has been recently shown
by Gonz\'{a}lez Perez and Teissier in (\cite{GT}), and by Grigoriev and Milman in (\cite{GM}), that for the case of (not necessarily normal)
toric varieties of any dimension, the iteration of Nash modification can be translated into a purely combinatorial algorithm.

Here we follow the results proved in \cite{GM}. Let $\xi=\{\gamma_1,\ldots,\gamma_r\}\subset\Z^2$
be a set of monomial exponents of some toric surface $X$, i. e., $X$ is the Zariski closure in $\C^r$ of
$\{(x^{\gamma_1},\ldots,x^{\gamma_r})|x\in(\C^*)^2\}$, where $x^{\gamma_i}=x_1^ {\gamma_{i,1}}\cdot x_2^ {\gamma_{i,2}}$. Let
$S=\{\{i,j\}\subset\{1,\ldots,r\}|\det(\gamma_i\mbox{ }\gamma_j)\neq0\}\}$. Fix
$\{i_0,j_0\}\in S$ and let
\begin{align}
A_{i_0}(\xi)&=\{\gamma_k-\gamma_{i_0}|k\in\{1,\ldots,r\}\setminus\{i_0,j_0\},\mbox{ }\det(\gamma_k\mbox{ }\gamma_{j_0})\neq0\},\notag\\
A_{j_0}(\xi)&=\{\gamma_k-\gamma_{j_0}|k\in\{1,\ldots,r\}\setminus\{i_0,j_0\},\mbox{ }\det(\gamma_k\mbox{ }\gamma_{i_0})\neq0\}.\notag
\end{align}
Let $\xi_{i_0,j_0}=A_{i_0}(\xi)\cup A_{j_0}(\xi)\cup\{\gamma_{i_0},\gamma_{j_0}\}$ and
$S'=\{\{i,j\}\in S|(0,0)\notin\Co(\xi_{i,j})\}$, where $\Co(\xi_{i,j})$ denotes
the convex hull of $\xi_{i,j}$ in $\R^2$. Then it is proved in \cite{GM} (Section 4) that, if $(0,0)\notin\Co(\xi)$,
the affine charts of Nash modification of $X$ are given by the toric surfaces associated to the sets $\xi_{i,j}$
such that $\{i,j\}\in S'$. The iteration of this algorithm gives rise to a tree in which every branch corresponds
to the choices of $\{i,j\}\in S'$. A branch of the algorithm ends if the semigroup $\Z_{\geq0}\xi_{i,j}$ is generated by two elements.

We will prove the following result: Fix $L:\R^2\rightarrow\R$, $(x_1,x_2)\mapsto ax_1+bx_2$, where
$a,b\in\Z$ and $(a,b)=1$ (we allow $a=1$, $b=0$, and $a=0$, $b=1$), such that $L(\xi)\geq0$. Let
$\gamma_i$, $\gamma_j\in\xi$ be two elements such that $L(\gamma_i)\leq L(\gamma_k)$ for all $\gamma_k\in\xi$,
$L(\gamma_j)\leq L(\gamma_k)$ for all $\gamma_k\in\xi$ such that $\det(\gamma_i\mbox{ }\gamma_k)\neq0$ and such
that $\{i,j\}\in S'$. We say that $L$ chooses $\gamma_i$, $\gamma_j$, even though, as we will see later,
$\gamma_i$, $\gamma_j$ need not be uniquely determined by the above conditions. We will prove that the
iteration of the algorithm stops for all the possible choices of $L$. In addition, we give a bound (that depends
on $L$) on the number of steps required for the algorithm to stop. Of course, this result gives only some progress
towards the question of whether or not Nash modification resolves singularities of toric surfaces. 

Our result has the following interpretation in terms of valuations. Let $X$ be the affine toric surface determined by 
$\xi\subset\Z^2$ and let $\C(x_1,x_2)$ be its field of rational functions. Let $\nu:\C(x_1,x_2)\rightarrow\Gamma$ 
be any valuation centered on $X$ such that $\nu(x_1)\neq\lambda\nu(x_2)$ for all $\lambda\in\R\setminus\Q$. We will
see in Section 5 that these valuations determine linear transformations $L$ as above and such that the center of $\nu$ 
after successive Nash modifications belongs to one of the charts chosen by $L$. By the result, the branches determined 
by $L$ are finite and they end in a non-singular surface. This implies the following theorem:

\begin{teo}
Let $\nu:\C(x_1,x_2)\rightarrow\Gamma$ be any valuation centered on $X$ such that $\nu(x_1)\neq\lambda\nu(x_2)$ for all
$\lambda\in\R\setminus\Q$. Then a finite iteration of Nash modification gives local uniformization along $\nu$.
\end{teo}

In other words, the problem of local uniformization of toric surfaces by iterating Nash modification remains open
only for the valuations $\nu$ of real rank 1 and rational rank 2, such that there exists $\lambda\in\R\setminus\Q$ such 
that $\nu(x_1)=\lambda\nu(x_2)$. I would like to mention that I have been informed that Pedro Gonz\'{a}lez Perez 
and Bernard Teissier have obtained a similar result for toric varieties of any dimension.

Finally, I would like to express my sincere gratitude to Mark Spivakovsky, whose constant support and guidance have been 
of great help to obtain the results presented here. Among other things, he guided me through the interpretation of the 
result in terms of valuations. I would also like to thank the referees for their careful reading and helpful comments that 
improved the presentation of the paper.




\section{\large The algorithm}

In this section we give an explicit description of the algorithm as stated in \cite{GM}
and of the concrete affine charts of the Nash modification of a toric surface that we will
follow.

\begin{defi}
Let $X\subset\C^r$ be an algebraic variety of pure dimension $m$. Consider the Gauss map:
\begin{align}
G:X\setminus\Si(X&)\rightarrow G(m,r)\notag\\
x&\mapsto T_xX,\notag
\end{align}
where $G(m,r)$ is the Grassmanian parameterizing the m-dimensional vector spaces in $\C^r$, and $T_xX$ is the
direction of the tangent space to $X$ at $x$. Denote by $X^*$ the Zariski closure of the graph of $G$. Call $\nu$ 
the restriction to $X^*$ of the projection of $X\times G(m,r)$ to $X$. The pair $(X^*,\nu)$ is called the Nash 
modification of $X$.
\end{defi}

We next define our main object of study.

\begin{defi}
Let $\xi:=\{\gamma_1,\ldots,\gamma_r\}\subset\Z^2$ such that $\Z\xi:=\{\sum_{k=1}^{r}\lambda_k\gamma_k|\lambda_k\in\Z\}=\Z^2$.
Consider the following monomial map:
\begin{align}
\Phi:(\C^*)&^2\rightarrow \C^r\notag\\
x=(x_1,x_2&)\mapsto (x^{\gamma_1},\ldots,x^{\gamma_r}),\notag
\end{align}
where $\C^*=\C\setminus\{0\}$, $x^{\gamma_k}:=x_1^{\gamma_{k,1}}\cdot x_2^{\gamma_{k,2}}$ for $k=1,\ldots,r$
and $\gamma_k=(\gamma_{k,1},\gamma_{k,2})$. Let $X$ denote the Zariski closure
of the image of $\Phi$. We call $X$ an affine toric variety and $\xi$ a set of monomial
exponents of $X$.
\end{defi}

It is known that $X$ is an irreducible surface that contains an algebraic group isomorphic
to $(\C^*)^2$ that extends to $X$ the natural action on itself (see \cite{St}, Chapter 13).
\\
\\
The following is a step-by-step description of the Nash modification algorithm for toric
surfaces as proved in \cite{GM} (Section 4):

\begin{itemize}
\item[(A1)]Let $\xi:=\{\gamma_1,\ldots,\gamma_r\}\subset\Z^2$ be a set of monomial exponents of some
toric surface $X$ such that $(0,0)\notin\Co(\xi)$.
\item[(A2)]Let $S:=\{\{i,j\}\subset\{1,\ldots,r\}|\det(\gamma_i\mbox{ }\gamma_j)\neq0\}\}$. Fix
some $\{i_0,j_0\}\in S$ and consider the sets
\begin{align}
A_{i_0}(\xi)&:=\{\gamma_k-\gamma_{i_0}|k\in\{1,\ldots,r\}\setminus\{i_0,j_0\},\mbox{ }
\det(\gamma_k\mbox{ }\gamma_{j_0})\neq0\},\notag\\
A_{j_0}(\xi)&:=\{\gamma_k-\gamma_{j_0}|k\in\{1,\ldots,r\}\setminus\{i_0,j_0\},\mbox{ }
\det(\gamma_k\mbox{ }\gamma_{i_0})\neq0\}.\notag
\end{align}
\item[(A3)]Consider $\xi_{i_0,j_0}:=A_{i_0}(\xi)\cup A_{j_0}(\xi)\cup\{\gamma_{i_0},\gamma_{j_0}\}$.
If $(0,0)\notin\Co(\xi_{i_0,j_0})$, then this set is a set of monomial exponents for one affine chart
of the Nash modi\-fication of $X$.

\begin{figure}[ht]
\begin{center}
\includegraphics{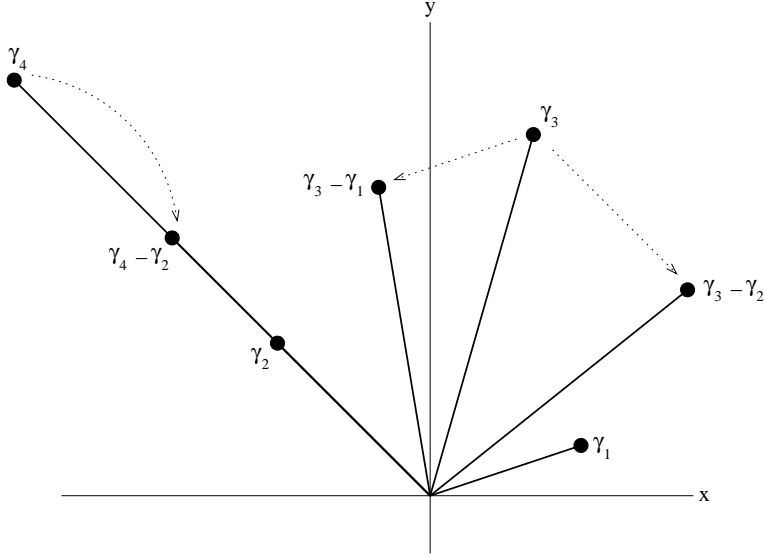}
\caption{Step (A3) of the algorithm for $\{1,2\}\in S$.}
\end{center}
\end{figure}

\item[(A4)]If the semigroup $\Z_{\geq0}\xi_{i_0,j_0}$ is generated by two elements then this
affine chart is non-singular and we stop. Otherwise, replace $\xi$ by $\xi_{i_0,j_0}$ and repeat
the process.
\end{itemize}

\begin{rem}
Notice that we can choose any set of generators $\xi'$ of the semigroup $\Z_{\geq0}\xi_{i_0,j_0}$ since
the resulting toric surfaces will be isomorphic. Moreover, if $\Z_{\geq0}\xi'=\Z_{\geq0}\xi_{i_0,j_0}$
then it is also clear that $\Z\xi'=\Z^2$. We say that $\xi'\subset\Z_{\geq0}\xi$ is a minimal set of
monomial exponents if $\xi'$ generates $\Z_{\geq0}\xi$ as a semigroup and for all $\gamma\in\xi'$,
$\gamma\notin\Z_{\geq0}(\xi'\setminus\{\gamma\})$.
\end{rem}

In this paper, we will only consider the elements of $S$ obtained in the following way:

\begin{itemize}
\item[(B1)]Fix any linear transformation $L:\R^2\rightarrow\R$, $(x_1,x_2)\mapsto ax_1+bx_2$, $a$, $b\in\Z$,
and $(a,b)=1$ (we allow $a=1$, $b=0$, and $a=0$, $b=1$), such that $L(\xi)\geq0$. We call $L(\gamma)$ the
$L-value$ of $\gamma$.
\item[(B2)]Let $\gamma_i$, $\gamma_j\in\xi$ be two elements such that $\{i,j\}\in S$, $L(\gamma_i)\leq L(\gamma_k)$ for all
$\gamma_k\in\xi$, $L(\gamma_j)\leq L(\gamma_k)$ for all $\gamma_k\in\xi$ such that
$\det(\gamma_i\mbox{ }\gamma_k)\neq0$, and such that $(0,0)\notin\Co(\xi_{i,j})$.
We say that $L$ \textit{chooses} $\gamma_i$ and $\gamma_j$.
\end{itemize}

\begin{rem}
For any $L$ satisfying $L(\xi)\geq0$, there exist $\gamma_i$, $\gamma_j\in\xi$ such that $(B2)$ holds. To see this, we consider 
four cases:
\begin{itemize}
\item[(1)]There exist two points $\gamma_1$, $\gamma_2$ such that $\det(\gamma_1\mbox{ }\gamma_2)\neq0$, $L(\gamma_1)<L(\gamma)$
for all $\gamma\in\xi\setminus\{\gamma_2\}$, and $L(\gamma_2)<L(\gamma)$ for all $\gamma\in\xi$ such that $\det(\gamma\mbox{ }\gamma_1)\neq0$.
Then $\gamma_1$, $\gamma_2$ satisfy $(B2)$.
\item[(2)]There exist at least one element of $L-$value 0. Among these points consider the one closest to the origin and call it
$\gamma$. Now consider the points in $\xi$ of lowest positive $L-$value. Among these points there is exactly one point $\gamma'$
such that $\gamma$, $\gamma'$ satisfy $(B2)$.
\item[(3)]$L(\xi)>0$ and there exist at least three elements $\gamma_1,$ $\gamma_2$, $\gamma_3$ such that 
$0<L(\gamma_1)=L(\gamma_2)=L(\gamma_3)\leq L(\gamma'),$ for all $\gamma'\in\xi\setminus\{\gamma_1,\gamma_2,\gamma_3\}$. Consider the
segment joining the points of $L-$value $L(\gamma_1)$. Then only the two couples consisting of one extremity of the segment and the point
next to it satisfy $(B2)$ (see figure \ref{f cases i ii}). 
\item[(4)]$L(\xi)>0$, there exists $\gamma\in\xi$ such that $0<L(\gamma)<L(\gamma')$ for all $\gamma'\in\xi\setminus\{\gamma\}$, and there 
are at least two elements $\gamma_1,$ $\gamma_2$, with both $\det(\gamma\mbox{ }\gamma_i)\neq0$ and such that 
$L(\gamma)<L(\gamma_1)=L(\gamma_2)\leq L(\gamma'),$ for all $\gamma'\in\xi$ such that $\det(\gamma\mbox{ }\gamma')\neq0$.
Then only the two couples consisting of $\gamma$ and one extremity of the segment joining the points of $L-$value $L(\gamma_1)$ satisfy 
$(B2)$ (see figure \ref{f cases i ii}).
\end{itemize}
\end{rem}

\begin{rem}
As we will see later, the choices of $L$ in $(B2)$ may not be unique (cf. lemma \ref{l cases i ii}). In addition, multiplying $L$ by a positive
constant does not modify its choices.
\end{rem}

\begin{exam}\label{e. ejemplos}
Let $\gamma_1=(1,0),\gamma_2=(2,1),\gamma_3=(0,2),\gamma_4=(0,3)$.
\begin{itemize}
\item[(A1)]Let $\xi=\{\gamma_1,\gamma_2,\gamma_3,\gamma_4\}\subset\Z^2$. Then
$S=\{\{1,2\},\{1,3\},\{1,4\},\{2,3\}\,\{2,4\}\}$.
\item[(B1)]Consider the following linear transformations:
\begin{itemize}
\item[(i)]$L_1(x,y)=y$.
\item[(ii)]$L_2(x,y)=\sqrt3 x+y$.
\end{itemize}
\item[(B2)]
\begin{itemize}
\item[(i)]$L_1$ chooses $\gamma_1$ and $\gamma_2$.
\item[(ii)]$L_2$ chooses $\gamma_1$ and $\gamma_3$.
\end{itemize}
\item[(A2)]For the choices $\{1,2\}$, $\{1,3\}$ we obtain, respectively:
\begin{itemize}
\item[(i)]$A_1(\xi)=\{\gamma_3-\gamma_1,\gamma_4-\gamma_1\}$,
$A_2(\xi)=\{\gamma_3-\gamma_2,\gamma_4-\gamma_2\}$.
\item[(ii)]$A_1(\xi)=\{\gamma_2-\gamma_1\}$, $A_3(\xi)=\{\gamma_2-\gamma_3,\gamma_4-\gamma_3\}$.
\end{itemize}
\item[(A3)]The resulting sets are, respectively:
\begin{itemize}
\item[(i)]$\xi_{1,2}=\{(-1,2),(-1,3)\}\cup\{(-2,1),(-2,2)\}\cup\{(1,0),(2,1)\}$.
\item[(ii)]$\xi_{1,3}=\{(1,1)\}\cup\{(2,-1),(0,1)\}\cup\{(1,0),(0,2)\}$.
\end{itemize}
\item[(A4)]The semigroups $\Z_{\geq0}\xi_{1,2}$, $\Z_{\geq0}\xi_{1,3}$ are generated by, respectively:
\begin{itemize}
\item[(i)]$\{(-2,1),(1,0)\}$. Therefore the algorithm stops for $L_1$.
\item[(ii)]$\{(0,1),(1,0),(2,-1)\}$. Replacing $\xi$ by $\xi_{1,3}$, we have to repeat the process for $L_2$. The algorithm
stops in the next iteration.
\end{itemize}
\end{itemize}
\end{exam}

What we intend to prove is that the algorithm stops for any choice of linear transformation
such that its kernel has rational slope or infinite slope. In other words, we will show that
in this case it is always possible to obtain a semigroup generated by two elements after
iterating the algorithm enough times.




\section{\large A first case}

In this section we study a first case of the problem stated in the previous section. Consider a
set of monomial exponents given by $\xi=\{(1,0),\gamma_1,\ldots,\gamma_r\}\subset\Z\times\Z_{\geq0}$.
We will iterate the algorithm following the choices of the linear transformation $L(x,y)=y$ and
we show that one eventually arrives to a semigroup generated by two elements (actually, those elements will
be $(1,0)$ and $(\lambda,1)$ for some $\lambda\in\Z$).
\\
\\
We intend to prove (always by following $L(x,y)=y$):
\begin{itemize}
\item[(1)]If $\xi=\{(1,0),(a_1,b_1),\ldots,(a_r,b_r)\}\subset\Z^2$ is such that
\begin{itemize}
\item[(i)]$\Z\xi=\Z^2$,
\item[(ii)]$b_i>1$ for all $i$,
\end{itemize}
then by iterating the algorithm we eventually arrive to an element of the form $(\lambda,1)$ which can be
taken by a linear isomorphism (that preserves $L$) to $(0,1)$.
\item[(2)]If $\xi=\{(1,0),(0,1),(-a_1,b_1),\ldots,(-a_r,b_r)\}$ is a minimal set of monomial exponents of
some toric surface where (necessarily, possibly after renumbering) $1\leq a_1<a_2<\ldots<a_r$ and $1<b_1<b_2<...<b_r$,
then by iterating the algorithm one eventually arrives to a semigroup generated by two elements.
\end{itemize}

Therefore, (1) implies that whenever $(1,0)\in\xi$ we can also suppose that $(0,1)\in\xi$, i. e.,
the situation in (2). 

\begin{rem}
The isomorphism that we will apply in (1) will be an element of $SL(2,\Z)$ that preserves $L$. The fact of being 
an isomorphism preserving $L$ guarantees that the algorithm is not modified. In addition, being an isomorphism 
guarantees that any relation among the elements of $\xi$ is preserved after applying it and no new relations appear. 
This means that the surfaces obtained after applying the Nash modification to isomorphic toric surfaces are also
isomorphic.
\end{rem}

\begin{lem}\label{l obtain 1}
For $\xi=\{(1,0),(a_1,b_1),(a_2,b_2),\ldots,(a_r,b_r)\}$ as in $(1)$, the iteration of the algorithm
eventually produces an element of the form $(\lambda,1)$, which can be taken by a linear isomorphism 
(that preserves $L$) to $(0,1)$.
\end{lem}
\begin{proof}
Since $\Z\xi=\Z^2$ we have $\gcd(b_1,b_2,\ldots,b_r)=1$ and we assume that $1<b_1<b_2<\cdots<b_r$. We
can assume this since if there were two points with the same $L-$value then one of them would be generated
by the other and some multiple of $(1,0)$. In addition, this property remains true after applying the algorithm
because the first choice of $L$ is always $(1,0)$.

Call $\gamma_0=(1,0)$ and $\gamma_i=(a_i,b_i)$. Then $L$ chooses $\gamma_0$ and  $\gamma_1$ and applying once
the algorithm we replace $\xi$ by $\xi_{0,1}$. As before, we consider a subset of $\xi_{0,1}$ that includes only
one element for every possible $L-$value (see figure \ref{f one elem for every L-value}). Of course, this is
a set of generators of $\Z_{\geq0}\xi_{0,1}$. Then we repeat the process taking into account this consideration.
Having this in mind, now it suffices to study the effect of the algorithm on the second coordinate.

\begin{figure}[ht]
\begin{center}
\includegraphics{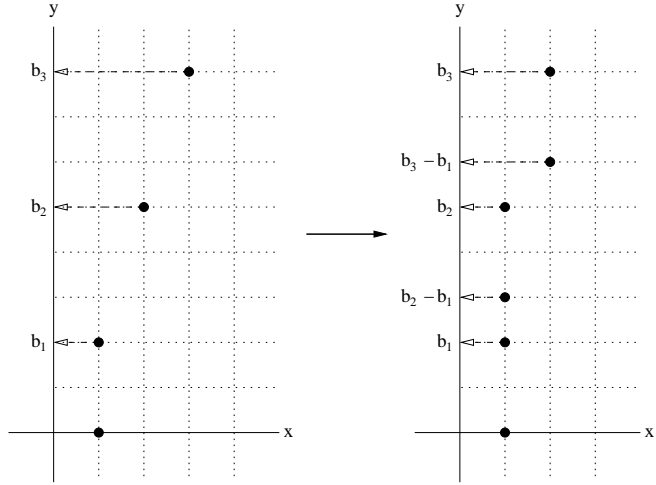}
\caption{One element for every possible $L-$value.\label{f one elem for every L-value}}
\end{center}
\end{figure}

We begin with integers $1<b_1<b_2<\cdots<b_r$ such that $\gcd(b_1,b_2,\ldots,b_r)=1$. After applying once
the algorithm we obtain a new set $\xi'$ such that the set of its elements' second coordinates contains the
subset $\{b_1,b_2-b_1,\ldots,b_r-b_1\}$. Since $\gcd(b_1,b_2-b_1,\ldots,b_r-b_1)=1$ we still have that the
greatest common divisor of the second coordinate of all points in $\xi'$ is 1. We repeat the algorithm until
we find some $n_1\in\N$ such that $b_2-n_1b_1\leq b_1$ and $b_2-(n_1-1)b_1>b_1$.

If $b_2-n_1b_1=b_1$ then $b_2$ is a multiple of $b_1$ and nothing happens. We keep repeating the algorithm until
we find some $n_2\in\N$ such that $b_3-n_2b_1\leq b_1$ and $b_3-(n_2-1)b_1>b_1$. Again, if $b_3-n_2b_1=b_1$ then $b_3$ is a multiple of $b_1$.
This situation cannot continue for all $b_i$ since $\gcd(b_1,b_2,\ldots,b_r)=1$. Therefore,
$b_i-nb_1<b_1$ for some $2\leq i\leq r$ and some $n\in\N$. At this moment, we have a new set $\xi'$ with some
element whose second coordinate is smaller than $b_1$ and such that the greatest common divisor of the second coordinate of all its elements is 1,
that is, we are in the same situation we began with.

Since all numbers involved are integers, this process will take us eventually to 1, that is, we will obtain
an element of the form $(\lambda,1)$, with $\lambda\in\Z$. Applying the linear isomorphism $T(x,y)=(x-\lambda y,y)$
we finally have $T(\lambda,1)=(0,1)$ and $T(1,0)=(1,0)$.
\end{proof}

Notice that when $r=2$ in the previous lemma the result of the algorithm on the second coordinate is precisely
Euclid's algorithm for $b_1$ and $b_2$. This observation directly implies the lemma in this case. We now proceed
to prove (2).

\begin{lem}
Let $\xi=\{(1,0),(0,1),(-a,b)\}$ where $a\geq1$ and $b>1$. Then the iteration of the algorithm eventually
produces a semigroup generated by two elements.
\end{lem}
\begin{proof}
We prove by induction that after applying the algorithm $n$ times where $n<b$ one obtains:
$$\{\delta_{n,i}|i=0,1,\ldots,n\}\cup\{e_1,e_2\},$$
where $\delta_{n,i}:=(-a-(n-i),b-i)$, $e_1=(1,0)$, and $e_2=(0,1)$ (see figure \ref{f resulting set}). Let $n=1$.
Since $b>1$, $L$ chooses $e_1$ and $e_2$. Then the algorithm gives $\{(-a-1,b),(-a,b-1\}\cup\{e_1,e_2\},$ which is
precisely $\{\delta_{1,i}|i=0,1\}\cup\{e_1,e_2\}.$

\begin{figure}[ht]
\begin{center}
\includegraphics{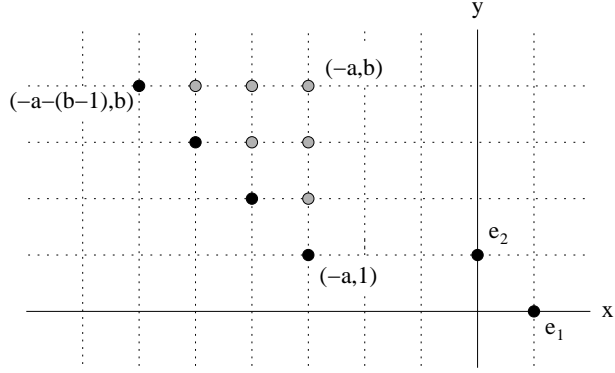}
\caption{The resulting set.\label{f resulting set}}
\end{center}
\end{figure}

Suppose that the statement is true for $n-1$. So, after applying the algorithm $n-1$ times, we obtain:
$$\{\delta_{n-1,i}|i=0,1,\ldots,n-1\}\cup\{e_1,e_2\}.$$
Since $n-1<b$, $L$ chooses again $e_1$ and $e_2$. Apply the algorithm again. Since $\det(\delta_{n-1,i}\mbox{ }e_1)\neq0$
and $\det(\delta_{n-1,i}\mbox{ }e_2)\neq0$ one takes $\{\delta_{n-1,i}-e_1|i=0,1,\ldots,n-1\}$ and $\{\delta_{n-1,i}-e_2|i=0,1,\ldots,n-1\}$.
But $\delta_{n-1,i}-e_1=\delta_{n,i}$ and $\delta_{n-1,i}-e_2=\delta_{n,i+1}$, which completes the induction.
In particular, for $n=b-1$ we obtain the set:
$$\xi'=\{(-a-(b-1),b),(-a-(b-2),b-1),\ldots,(-a,1)\}\cup\{e_1,e_2\}.$$
Notice that the points $(-a-(n-i),b-i)$ for $i=0,1,\ldots,n$ are all contained in
some line $l_n$ of slope -1, for each $n$. Now, since $-\frac{1}{a}\geq-1$, this implies, for $n=b-1$,
that every point in $\xi'$ is generated by $(-a,1)$ and $(1,0)$.
Therefore, after $b-1$ steps, the resulting semigroup is generated by two elements.
\end{proof}

\begin{pro}\label{p first case}
Let $\xi=\{(1,0),(0,1),(-a_1,b_1),\ldots,(-a_r,b_r)\}$, where $1\leq a_1<a_2<\ldots<a_r$ and
$1<b_1<\ldots<b_r$, be as in (2). Then the iteration of the algorithm eventually produces a semigroup generated
by two elements.
\end{pro}
\begin{proof}
We proceed by induction on the number of elements of $\xi$. The case $r=1$ is given by the previous lemma.
Assume that the result holds for $r-1$. As in the previous lemma, after applying
the algorithm $b_1-1$ times every $(-a_j,b_j)$ gives rise to (see figure \ref{f resulting sets}):
$$\xi'_j:=\{(-a_j-(b_1-1-i),b_j-i)|i=0,1,\ldots,b_1-1\}.$$

\begin{figure}[ht]
\begin{center}
\includegraphics{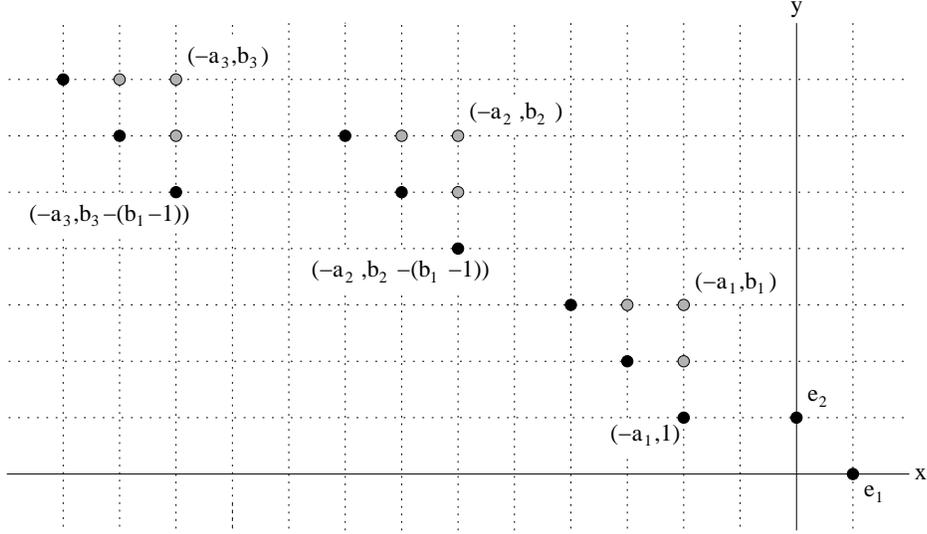}
\caption{The resulting sets.\label{f resulting sets}}
\end{center}
\end{figure}

As before, each $\xi'_j$ is contained in some line of slope -1. Therefore, since $-\frac{1}{a_1}\geq-1$,
every element in $\xi'_j$ is generated by $(-a_j,b_j-(b_1-1)),\mbox{ }(-a_1,1)$, and $(1,0)$ for each $j$.
Therefore, if $\xi'=\{(-a_i,b_i-(b_1-1))|i=2,\ldots,r\}\cup\{(-a_1,1),(1,0)\}$, we have
$$\bigcup_{j=1}^{r}\xi'_j\cup\{(1,0),(0,1)\}\subset\Z_{\geq0}\xi'.$$
Next, we consider the linear isomorphism $T(x,y)=(x+a_1y,y)$. Then we have
(since $T(-a_1,1)=(0,1)$),
$$T(\xi')=\{(1,0),(0,1),(-c_2,d_2),(-c_3,d_3),\ldots,(-c_r,d_r)\}.$$
Since $|\xi|=r+2$ and $|T(\xi')|=r+1$  we have, by induction, that the iteration of the algorithm over
$\xi$ eventually produces a semigroup generated by two elements.
\end{proof}

\begin{rem}
Notice that if $\xi=\{(-1,0),(a_1,b_1),\ldots,(a_r,b_r)\}$ then analogous results (1) and (2) for this set
can be reduced to the previous ones by considering the linear isomorphism $T(x,y)=(-x,y)$, since this isomorphism
preserves $L$.
\end{rem}




\section{\large Rational chart}

Consider any set of monomial exponents given by $\xi=\{\gamma_1,\ldots,\gamma_r\}\subset\Z^2$.
In this section we are going to prove that the iteration of the algorithm following $L:\R^2\rightarrow\R$,
$(x,y)\mapsto(ax+by)$ where $a$, $b\in\Z$ (which can be assumed to be relatively prime) and such that $L(\xi)\geq0$,
eventually produces a semigroup generated by two elements. To reach this goal, we intend to reduce this case
to the one already solved. Under these assumptions we can assume that $\xi\subset\Z\times\Z_{\geq0}$ and that $L(x,y)=y$
(it suffices to take the isomorphism $T(x,y)=(\beta x-\alpha y,ax+by)$, where $\alpha a+\beta b=1$).
\\
\\
We intend to prove (always by following $L(x,y)=y$):

\begin{itemize}
\item[(1)]If $\xi=\{\gamma_1,\ldots,\gamma_r\}\subset\Z^2$ such that $L(\gamma_i)>0$ for all $i$,
then by iterating the algorithm we eventually arrive to an element of the form $(n,0)$, with $n\in\Z$.
\item[(2)]If $\xi=\{(n,0),\gamma_1,\ldots,\gamma_r\}$ is a set of monomial exponents of
some toric surface with $n>0$, then the iteration of the algorithm eventually produces the point $(1,0)$.
\end{itemize}

\begin{lem}\label{l cases i ii}
If $\xi=\{\gamma_1,\ldots,\gamma_r\}\subset\Z^2$ such that $L(\gamma_i)>0$ for all $i$, then by
iterating the algorithm we eventually arrive to an element of the form $(n,0)$, with $n\in\Z$.
\end{lem}
\begin{proof}
First, notice that the choices of $L$ are not unique in the following cases (see figure \ref{f cases i ii}):

\begin{itemize}
\item[(i)]There exist at least three elements $\gamma_1,$ $\gamma_2$, $\gamma_3$
such that $$0<L(\gamma_1)=L(\gamma_2)=L(\gamma_3)\leq L(\gamma'),$$
for all $\gamma'\in\xi\setminus\{\gamma_1,\gamma_2,\gamma_3\}$.
\item[(ii)]There exists $\gamma\in\xi$ such that $0<L(\gamma)<L(\gamma')$ for all
$\gamma'\in\xi\setminus\{\gamma\}$ and there are at least two elements $\gamma_1,$ $\gamma_2$, with both
$\det(\gamma\mbox{ }\gamma_i)\neq0$ and such that $$0<L(\gamma)<L(\gamma_1)=L(\gamma_2)\leq L(\gamma'),$$
for all $\gamma'\in\xi$ such that $\det(\gamma\mbox{ }\gamma')\neq0$.
\end{itemize}

\begin{figure}[ht]
\begin{center}
\includegraphics{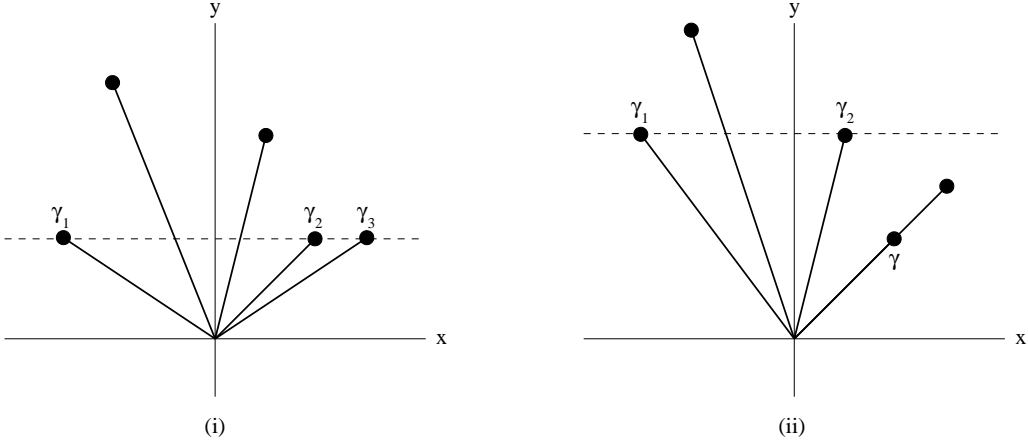}
\caption{Cases (i) and (ii).\label{f cases i ii}}
\end{center}
\end{figure}

In addition, an element of $L-$value 0 could be obtained only after being in one of the cases (i) or (ii).
Suppose first that we are not in any of the cases above, i.e. $\xi$ does not satisfy either (i) or (ii).
Now, let us suppose (possibly after renumbering) that $L(\gamma_i)\leq L(\gamma_r),$ for all $1\leq i\leq r$ 
and that $L$ chooses $\gamma_1$ and $\gamma_2$. Apply the algorithm once to obtain
$\xi'=\{\gamma'_1,\ldots,\gamma'_{r'}\}$. Since $\xi$ does not satisfy either (i) or (ii), we have
$0<L(\gamma'_i)$ for all $1\leq i\leq r'$. Once again, possibly after renumbering, we have $L(\gamma'_i)\leq L(\gamma'_{r'}),$ 
for all $1\leq i\leq r'$. Then, $\gamma'_{r'}=\gamma_i-\gamma_j$ for some $i>2$ and some $j\in\{1,2\}$, or $\gamma'_{r'}=\gamma_j$
for some $j\in\{1,2\}$. If $\gamma'_{r'}=\gamma_j$ then $L(\gamma'_{r'})=L(\gamma_j)<L(\gamma_{r})$. This inequality is strict since 
$\xi$ does not satisfy either (i) or (ii). If $\gamma'_{r'}=\gamma_i-\gamma_j$ for some $i>2$ and some $j\in\{1,2\}$ then
$$L(\gamma'_{r'})=L(\gamma_i)-L(\gamma_j)<L(\gamma_i)\leq L(\gamma_{r}).$$
If $\xi'$ does not satisfy either (i) or (ii) then we are in the same situation we begin with but now 
$L(\gamma'_{r'})<L(\gamma_{r})$. Since $L(\gamma)\in \N$ this situation cannot continue infini-tely many times. 
Therefore, either the resulting semigroup after some iteration of the algorithm is generated by two elements or
we arrive at one of the cases (i) or (ii).
\\
\\
So suppose we are in case (i). Let $k:=L(\gamma_1)=L(\gamma_2)=L(\gamma_3)$. Denote by
$\{\rho_1,\ldots,\rho_s\}$ all the elements of $\xi$ whose $L-$value is $k$. We can
suppose that $c_x(\rho_1)<c_x(\rho_2)<\ldots<c_x(\rho_s)$, where $c_x(\rho_i)$ denotes the
first coordinate of $\rho_i$. Under these assumptions, $L$ may choose only the couples
$\{\rho_1,\rho_2\}$ or $\{\rho_{s-1},\rho_s\}$. Indeed, let us suppose that $L$ chooses
$\{\rho_i,\rho_j\}$ different from $\{\rho_1,\rho_2\}$ and $\{\rho_{s-1},\rho_s\}$. If $s=3$, then
$\{\rho_i,\rho_j\}=\{\rho_1,\rho_3\}$. This implies that, after applying the algorithm, $c_x(\rho_2-\rho_1)>0$ and
$c_x(\rho_2-\rho_3)<0$ and then $(0,0)\in\Co(\rho_2-\rho_1,\rho_2-\rho_3)\subset\Co(\xi')\subset\R^2$,
where $\xi'$ is the resulting set after applying the algorithm. But according to (B2) of the
algorithm, we are supposed to choose only couples such that $(0,0)\notin\Co(\xi')$, that is, we have
a contradiction. If $s>3$, reasoning similarly we have the same conclusion. So let us suppose that
$L$ chooses the couple $\{\rho_1,\rho_2\}$. Applying the algorithm one more time will give us
$0<c_x(\rho_i-\rho_1)$, $0<c_x(\rho_i-\rho_2)$, $L(\rho_i-\rho_1)=0$, and $L(\rho_i-\rho_2)=0$
for all $i>2$. Since $s\geq3$ we have at least one element in the resulting set whose $L-$value
is 0, which in this case has the form $(n,0)$ with $n>0$. If $L$ chooses the couple $\{\rho_{s-1},\rho_s\}$
then we will obtain an element of the form $(m,0)$ with $m<0$.
\\
\\
Now suppose that we are in case (ii). Let $k:=L(\gamma_1)=L(\gamma_2)$. We denote
by $\{\rho_1,\ldots,\rho_s\}$ all the elements of $\xi$ whose $L-$value is $k$. Once again, we can
suppose that $c_x(\rho_1)<c_x(\rho_2)<\ldots<c_x(\rho_s)$. Reasoning as before $L$ chooses
$\gamma$ and could choose only $\rho_1$ or $\rho_s$. Let us suppose that $L$ chooses $\rho_1$.
Then $0<c_x(\rho_i-\rho_1)$ and $L(\rho_i-\rho_1)=0$ for all $i>1$ such that $\det(\rho_i\mbox{ }\gamma)\neq0$.
If $L$ chooses $\gamma$ and $\rho_s$ the result is analogous. Since $s\geq2$ we have at least one element in
the resulting set whose $L-$value is 0 which is what we wanted to prove.
\end{proof}

Now we proceed to prove (2).

\begin{lem}
If $\xi=\{(n,0),\gamma_1,\ldots,\gamma_r\}$ is a set of monomial exponents of some toric surface with $n>0$,
then the iteration of the algorithm eventually produces a point of the form $(\lambda,1)$, where $\lambda\in\Z$.
\end{lem}
\begin{proof}
Denote by $\{(n,0),\rho_1,\ldots,\rho_s\}$ the elements of $\xi$ whose $L-$value is 0 and
suppose that $0<n<c_x(\rho_i)$ for all $i$. Then $L$ first chooses $(n,0)$. Otherwise,
since $L(\rho_i)=0$ for all $i$, $L$ is forced to choose some of the $\rho_i$, and we
would have $c_x(\rho_i-(n,0))<0$ which contradicts condition (B2). Therefore $L$
chooses $(n,0)$. The other possible point should be then the one whose first coordinate
is the smallest among all the points in the next value of $L$.

Denote by $\{\sigma_1,\ldots,\sigma_t\}$ the elements of $\xi$ whose $L-$value
is greater than 0 and suppose that $0<L(\sigma_1)\leq L(\sigma_2)\leq\ldots\leq L(\sigma_t)$ and
that $L$ chooses $\sigma_1$. Since $\Z\xi=\Z^2$, we have $\gcd(L(\sigma_1), L(\sigma_2),\ldots,L(\sigma_t))=1$.
Apply once the algorithm. Then we obtain a new set $\xi'$ that contains the subset
$\{\sigma_1,\sigma_2-\sigma_1,\ldots,\sigma_t-\sigma_1\}$ (see figure \ref{f looking for lambda 1}).

\begin{figure}[ht]
\begin{center}
\includegraphics{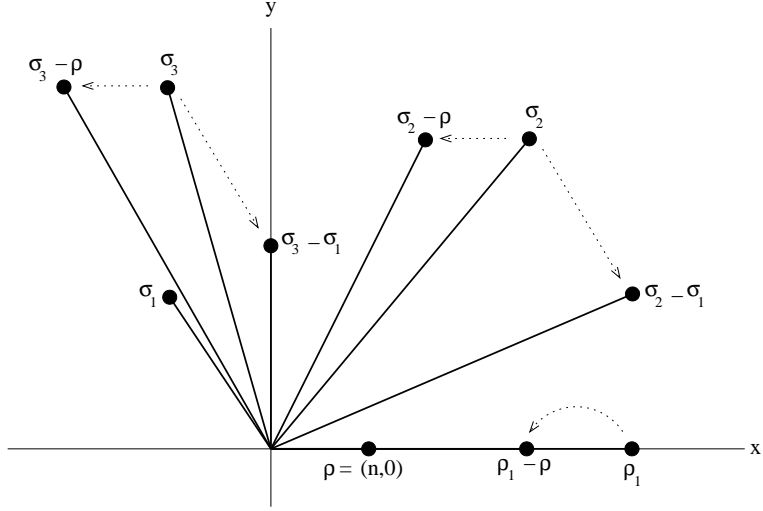}
\caption{Looking for $(\lambda,1)$.\label{f looking for lambda 1}}
\end{center}
\end{figure}

Since $\gcd(L(\sigma_1),L(\sigma_2)-L(\sigma_1),\ldots,L(\sigma_t)-L(\sigma_1))=1$,
we still have that the greatest common divisor of the $L-$values of all points in $\xi'$ is 1.
As we did in lemma \ref{l obtain 1}, we continue applying the algorithm until we have $L(\sigma_i)-mL(\sigma_1)<L(\sigma_1)$
for some $2\leq i\leq t$ and some $m\in\N$. At this moment, we have a new set of monomial exponents with some element whose
$L-$value is smaller than $L(\sigma_1)$ and such that the greatest common divisor of the $L-$values of all
its elements is 1, that is, we are in the same situation we began with. Continuing this way, we eventually
obtain the desired point. Once we get to some point (or points) whose $L-$value is 1, then the one with smallest
first coordinate is not generated by the others. As in lemma \ref{l obtain 1}, we can assume that this point is $(0,1)$.
\end{proof}

This lemma allows us to assume that $(0,1)\in\xi$. The next proposition shows that we can obtain some $(m,0)$
in the resulting set after applying the algorithm enough times such that $m<n$. Since there is always a point $(\lambda,1)$
at each step of the algorithm, we will have the same situation but with $m<n$. Continuing this
way we will eventually obtain the element $(1,0)$.

\begin{lem}\label{l rational chart}
Let $\xi=\{(n,0),(0,1),\gamma_1,\ldots,\gamma_r\}$ be a minimal set of monomial exponents of some toric surface,
where $n>0$. Then the iteration of the algorithm eventually produces the point $(1,0)$.
\end{lem}
\begin{proof}
Suppose that $(n,0)$ has the smallest first coordinate among all elements of $L$-value 0.
We want to find another element whose $L-$value is 1 and whose first coordinate is not a multiple of $n$.
Let $\xi_n:=\xi\cap (n\Z\times\Z)$ and $\xi_0:=\xi\setminus\xi_n$. Since $\xi$ is minimal, we may assume that $(0,1)$
is the only element of $L-$value 1 in $\xi_n$. Then $L$ chooses $(n,0)$ and $(0,1)$. If $\xi'$ is the resulting set
after applying the algorithm once, we have $(\xi_n)'=\xi'\cap (n\Z\times\Z)$ and $(\xi_0)'=\xi'\setminus(\xi_n)'$. In
other words, the elements in $\xi_n$ only produce elements in $n\Z\times\Z$ and the elements outside of $\xi_n$ only produce
elements outside of $n\Z\times\Z$. Therefore, as long as $L$ keeps choosing $(n,0)$ and $(0,1)$, the effect of the algorithm
on $\xi_n$ is precisely what we saw in proposition \ref{p first case} (see figure
\ref{f first case induction}). In addition, the linear isomorphism we used in that proposition,
$T(x,y)=(x-\lambda y,y)$, does not change this property if $\lambda$ is a multiple of $n$ since, in this case,
$T(\gamma)\in n\Z\times\Z$ if and only if $\gamma\in n\Z\times\Z$. All this implies that the effect of the algorithm
on $\xi_0$ is independent of the effect on $\xi_n$.

\begin{figure}[ht]
\begin{center}
\includegraphics{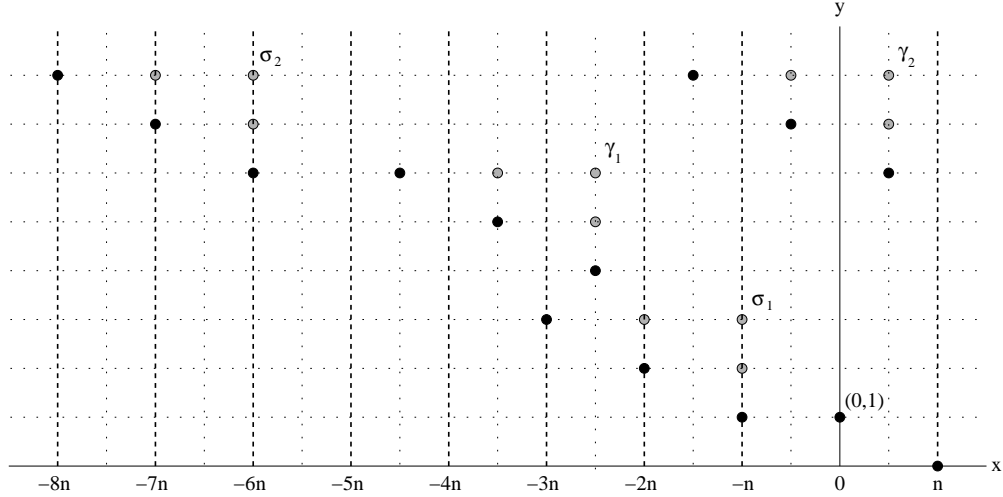}
\caption{$\sigma_1,$ $\sigma_2\in\xi_n$ and $\gamma_1,$ $\gamma_2\in\xi_0$.\label{f first case induction}}
\end{center}
\end{figure}

Now, since $\Z\xi=\Z^2$, there must exist some point $\gamma\in\xi$ such that $\gamma\notin n\Z\times\Z$.
Of all these possible elements we consider the one with smallest $L-$value and if there are several
such points, we take the one whose first coordinate is the smallest. Call this point $(a,b)$. We then apply the algorithm $b-1$ times.
If there is some point in $\xi_n$ whose $L-$value is smaller than $b$ then we will have to use the isomorphism $T(x,y)=(x-\lambda y,y)$ after
some iteration in order to obtain again the point $(0,1)$. As we said before, this does not change the evolution of the point $(a,b)$ or
its $L-$value. So, continuing this way, after these $b-1$ times, we obtain another element $(\lambda,1)$ different from $(0,1)$ and such 
that $\lambda$ is not a multiple of $n$.

At the next step, there will be some point $(m,0)$ different from $(n,0)$. If $m<n$ we finish. If not,
apply the algorithm again to obtain the point $(m-n,0)$. Continuing this way, since $m$ is not a multiple of $n$,
we eventually obtain some $(m',0)$ with $m'<n$. If in this process appears some other point $(p,0)$ such that $0<p<n$
or $n<p<m$ the conclusion is the same.
\end{proof}

\begin{rem}
Notice that if $\xi=\{(n,0),\gamma_1,\ldots,\gamma_r\}$ with $n<0$, then the analogous result (2) for this set can be reduced
to the case $n>0$ by considering the linear isomorphism $T(x,y)=(-x,y)$, since this isomorphism preserves $L$.
\end{rem}

Putting together the results (1) and (2) of this section and the previous one, we obtain that the iteration of the algorithm 
(A1) to (A4) subject to the rules (B1) and (B2) eventually stops.

\begin{teo}\label{t. main theorem}
Let $\xi\subset\Z^2$ be a set of monomial exponents of some toric surface. Then the iteration of the algorithm
following $L(x,y)=ax+by$, where $a$, $b\in\Z$, and $L(\xi)\geq0$, eventually produces a semigroup generated by two elements.
\end{teo}




\section{\large Counting steps}

In this section we are going to prove some results regarding the number of iterations that the algorithm
needs to stop in the cases we already solved.
\\
\\
Let $\xi=\{\gamma_1,\ldots,\gamma_r\}\subset\Z\times\Z_{\geq0}$ be a set of monomial exponents of some
toric surface and consider $L(x,y)=y$. Let
\begin{align}
u_0(\xi):=\max\{L(\gamma_i)|&\gamma_i\in\xi\}\notag\\
u_1(\xi):=\min\{L(\gamma_i)|&\gamma_i\in\xi\mbox{, }\Z(\gamma_{j_0},\ldots,\gamma_{j_{s}})=\Z^2\mbox{ where }
\{\gamma_{j_0},\ldots,\gamma_{j_{s}}\}\mbox{ denotes}\notag\\
&\mbox{ the set of all } \gamma_{j_k}\mbox{ such that }0\leq L(\gamma_{j_k})\leq L(\gamma_i)\}\notag
\end{align}

Suppose that $L(\gamma_i)>0$ for all $i$ and denote by $\xi_k$ the resulting set after applying the algorithm
$k$ times. Then we have the following two lemmas:

\begin{lem}\label{l u_0 iterations}
Suppose that after $u_0(\xi)$ iterations of the algorithm we obtain an element of $L-$value 0 for the first time.
Then
\begin{itemize}
\item[(1)]$0\leq L(\xi_{u_0(\xi)})\leq1$.
\item[(2)]There exists some $\gamma\in\xi_{u_0(\xi)}$ such that $L(\gamma)=1$.
\item[(3)]There exist $\gamma_1,\ldots,\gamma_t\in\xi_{u_0(\xi)}$ such that $L(\gamma_i)=0$, $t\geq2$, and
such that $\gcd(c_x(\gamma_1),\ldots,c_x(\gamma_t))=1$, where $c_x(\gamma_i)$ denotes the first coordinate of $\gamma_i$.
\end{itemize}
\end{lem}
\begin{proof}
Recall that an element of the form $(n,0)$ is produced only after being in one of the cases (i) or (ii)
of lemma \ref{l cases i ii}. The hypothesis means that only after $u_0(\xi)-1$ iterations we arrive to one of
these cases. Since after each iteration the value of $u_0(\cdot)$ decreases at least by one, after $u_0(\xi)-1$
iterations all points in the resulting set must have $L-$value 1. Another application of the algorithm gives us
(1) and (2) for any choice of couples of $L$. Let $\xi_{u_0(\xi)-1}=\{(a_1,1),(a_2,1),\ldots,(a_r,1)\}$, where
$a_1<a_2<\cdots<a_r$. Suppose that $L$ chooses $(a_1,1)$ and $(a_2,1)$. Then another application of the algorithm
produces
$$\{(a_3-a_1,0),\ldots,(a_r-a_1,0)\}\cup\{(a_3-a_2,0),\ldots,(a_r-a_2,0)\}\cup\{(a_1,1),(a_2,1)\}.$$
Then $\gcd(a_3-a_1,\ldots,a_r-a_1,a_3-a_2,\ldots,a_r-a_2)=1$. Indeed, since $\Z\xi_{u_0(\xi)-1}=\Z^2$
there exist some $\lambda_i\in\Z$ such that $\sum_{i=1}^{r}\lambda_i(a_i,1)=(1,0)$. Consider the linear isomorphism
$T(x,y)=(x-a_1y,y)$. Then $T(\sum_{i=1}^{r}\lambda_i(a_i,1))=T(1,0)=(1,0)$. In particular, $\sum_{i=2}^{r}\lambda_i(a_i-a_1)=1$,
i. e., $\gcd(a_2-a_1,\ldots,a_r-a_1)=1$. This implies the assertion. If $L$ chooses $(a_{r-1},1)$ and $(a_r,1)$, we proceed
similarly.
\end{proof}

For the next lemma, rename $\xi$ as $\xi_0$. Now suppose that after $w<u_0(\xi_0)$ iterations of the algorithm we obtain an
element of $L-$value 0 for the first time, and denote by $\xi=\{(n,0),\gamma_1,\ldots,\gamma_r\}$
the resulting set. Let us suppose that $L$ chooses $\gamma_0=(n,0)$ and $\gamma_1$, so, in particular,
$0=L(\gamma_0)<L(\gamma_1)\leq L(\gamma_j)$, for all $\gamma_j\in\xi$ such that $\det(\gamma_0\mbox{ }\gamma_j)\neq0$.

\begin{lem}\label{l less u_0 iterations}
Let $\xi'=\{\gamma'_1,\ldots,\gamma'_{r'}\}$ be the resulting set after applying the algorithm once again and suppose
that the semigroup $\Z_{\geq0}\xi'$ is not generated by two elements. If $n>0$ then:
\begin{itemize}
\item[(1)]If $L(\gamma_1)=u_1(\xi)$ then $L(\gamma_1)=1$ and $\xi'$ contains $(1,0)$ or at least two elements
of $L-$value 0 and whose first coordinates are relatively prime. In particular, $\xi$ contains an element of
$L-$value 1.
\item[(2)]If $L(\gamma_1)<u_1(\xi)$ then $u_1(\xi')<u_1(\xi)$.
\item[(3)]If the semigroup $\Z_{\geq0}\xi_{u_1(\xi)}$ is not generated by two elements, then $\xi_{u_1(\xi)}$ contains $(1,0)$
or at least two elements of $L-$value 0 whose first coordinates are relatively prime, and an element of $L-$value 1.
\end{itemize}
\end{lem}
\begin{proof}
Let $\xi^*=\{\gamma_{j_0},\ldots,\gamma_{j_{s}}\}$ be the elements $\gamma_j\in\xi$ such that
$0\leq L(\gamma_j)\leq u_1(\xi)$. Let $\gamma\in\xi^*$ be such that $L(\gamma)=u_1(\xi)$. Suppose that
$(n,0)=\gamma_{j_0}$, $\gamma_1=\gamma_{j_1}$, and $\gamma=\gamma_{j_{s}}$. By the definition of $u_1(\xi)$,
we have $\Z\xi^*=\Z^2$.
\begin{itemize}
\item[(1)]Suppose that $L(\gamma_1)=u_1(\xi)$. Then $L(\gamma_1)=L(\gamma)$ so
$L(\gamma_{j_k})=0$ or $L(\gamma_{j_k})=L(\gamma_1)$ for all $\gamma_{j_k}\in\xi^*$. Since
$\Z\xi^*=\Z^2$ we have $\gcd(L(\gamma_{j_0}),\ldots,L(\gamma_{j_{s}}))=1$. But then $L(\gamma_1)>0$ implies
$L(\gamma_1)=1$. If, in addition, $n=1$ then we are done.
Suppose $n>1$. Then the cardinality of $\xi^*$ is at least 3. Now proceed as in the previous lemma to find
the elements whose first coordinate are relatively prime.
\item[(2)]Suppose now that $L(\gamma_1)<u_1(\xi)$. Apply the algorithm once to obtain $\xi'$. Consider
the subset $$\xi'^*=\xi'_1\cup\xi'_2\cup\{\gamma_0,\gamma_1\},$$ where $\xi'_1=\{\gamma_i-\gamma_1|i\in\{j_2,\ldots,j_{s}\}\mbox{, }L(\gamma_i)>0\}$
and $\xi'_2=\{\gamma_i-\gamma_0|L(\gamma_i)=0\}$ (see figure \ref{f L<u_2}).
\begin{figure}[ht]
\begin{center}
\includegraphics{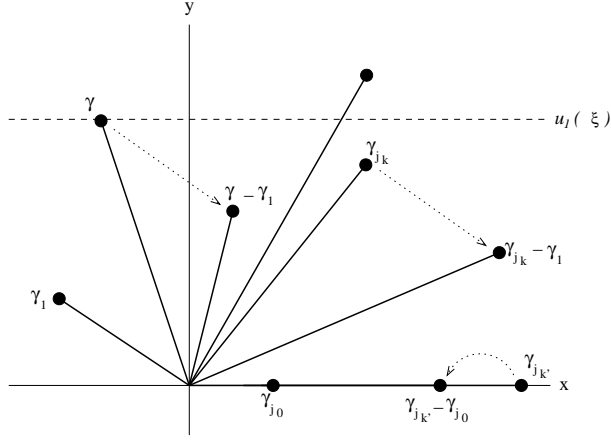}
\caption{$L(\gamma_1)<u_1(\xi)$.\label{f L<u_2}}
\end{center}
\end{figure}
Since $\xi^*\subset\Z\xi'^*$ then $\Z^2=\Z\xi^*\subset\Z\xi'^*$, that is, $\Z^2=\Z\xi'^*$.
Now consider $l=\max\{L(\gamma)-L(\gamma_1),L(\gamma_1)\}$. Since $l\geq L(\gamma_j)$ for all
$\gamma_j\in\xi'^*$ then $u_1(\xi')\leq l$. In addition, $l\leq L(\gamma)=u_1(\xi)$, so that
$$u_1(\xi')\leq u_1(\xi).$$
Suppose that $l=u_1(\xi')$. Then, if $l=L(\gamma)-L(\gamma_1)$ we have $u_1(\xi')=l<L(\gamma)=u_1(\xi)$, since
$L(\gamma_1)>0$. If $l=L(\gamma_1)$ we obtain the same conclusion since, by hypothesis, $L(\gamma_1)<L(\gamma)$.
So, if $l=u_1(\xi')$, for the two possible choices of $l$, we have $u_1(\xi')<u_1(\xi)$. Otherwise $u_1(\xi')<l$
and the conclusion follows once again.
\item[(3)]Since $1\leq u_1(\xi)$, then by $(2)$, after at most $u_1(\xi)-1$ iterations, we will obtain
$u_1(\cdot)=1$. Then by $(1)$ we conclude.
\end{itemize}
\end{proof}

\begin{rem}
The analogous result of the previous lemma for $n<0$ can be reduced to the case $n>0$ by considering the linear isomorphism
$T(x,y)=(-x,y)$, since this isomorphism preserves $L$.
\end{rem}

According to the previous results, after at most $u_0(\xi)$ iterations, the algorithm will produce, first,
an element $(n,0)$, then, some other points of $L-$value 0 such that their first coordinates are relatively prime.
Of all these points, call $(N,0)$ the one with biggest (or smallest if $n<0$) first coordinate. Our next goal will
be to find a bound for $N$.

\begin{lem}\label{l bound for n}
Let $\xi=\{\gamma_1,\ldots,\gamma_r\}\subset\Z^2$ be a set of monomial exponents of some toric surface such
that $L(\gamma_i)\geq0$ for all $i$. Let $v_0(\xi):=\max\{|c_x(\gamma_i)||\gamma_i\in\xi\}$. Let $\xi_w$ be
the resulting set after iterating the algorithm $w$ times. Then
$$v_0(\xi_w)\leq2^w\cdot v_0(\xi).$$
\end{lem}
\begin{proof}
We proceed by induction on $w$. For $w=1$ it is clear that $v_0(\xi_1)\leq2\cdot v_0(\xi)$ (see
figure \ref{f 2-pot-w}).
\begin{figure}[ht]
\begin{center}
\includegraphics{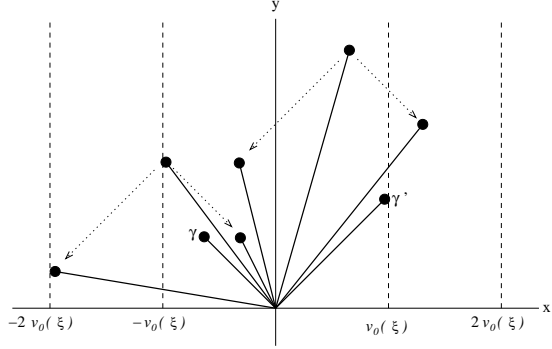}
\caption{$v_0(\xi_w)\leq2^w\cdot v_0(\xi)$.\label{f 2-pot-w}}
\end{center}
\end{figure}
Suppose that $v_0(\xi_k)\leq2^k\cdot v_0(\xi).$ This means that for all $\gamma\in\xi_k$ we have
$-2^k\cdot v_0(\xi)\leq c_x(\gamma)\leq2^k\cdot v_0(\xi)$, and this is true, in particular, for the two elements
chosen by $L$. Therefore, $v_0(\xi_{k+1})\leq2^k\cdot v_0(\xi)+2^k\cdot v_0(\xi)=2^{k+1}\cdot v_0(\xi)$,
which completes the induction.
\end{proof}

\begin{lem}\label{l bound for 1}
Let $\xi=\{(n_1,0),\ldots,(n_s,0)\}\cup\{\gamma_1,\ldots,\gamma_r\}$ be such that $0<L(\gamma_i)$ and $\gcd(n_1,\ldots,n_s)=1$.
Assume that $0<n_1<n_2<\cdots<n_s$. If $n_1=1$, put $v_1(\xi):=1$. If $n_1>1$ let
\begin{align}
v_1(\xi):=\min\{n_i|&\gcd(n_{j_1},\ldots,n_{j_t})=1 \mbox{ where } \{n_{j_1},\ldots,n_{j_t}\}\mbox{ denotes }\notag\\
&\mbox{ the set of all } n_{j_k}\mbox{ such that } n_{j_k}\leq n_i\}\notag
\end{align}
If $\xi'$ denotes the resulting set after applying the algorithm once, then $v_1(\xi')\leq v_1(\xi)-2$. Therefore, if $n_1>1$,
after at most $\lfloor\frac{v_1(\xi)}{2}\rfloor$ iterations we will obtain the element $(1,0)$.
\end{lem}
\begin{proof}
Since we are looking for the element $(1,0)$, we assume that $n_1>1$. Suppose that $n_{i_0}=v_1(\xi)$ where $2\leq i_0\leq s$.
After applying once the algorithm we obtain, in particular, the subset $\{(n_1,0),(n_2-n_1,0),\ldots,(n_{i_0}-n_1,0)\}\subset\xi'.$
Call $N=\max\{n_1,n_{i_0}-n_1\}$. Since $\gcd(n_1,n_2-n_1,\ldots,n_{i_0}-n_1)=1$ we have $v_1(\xi')\leq N$. If $N=n_{i_0}-n_1$
then, since $n_1\geq2$ we have $v_1(\xi')\leq n_{i_0}-n_1\leq v_1(\xi)-2$. Suppose now that $N=n_1$. If $n_{i_0}=n_1+1$ then
$n_{i_0}-n_1=1$ and $v(\xi')=1$ and we are done. Otherwise $n_{i_0}>n_1+1$ which implies $v_1(\xi')\leq n_1\leq n_{i_0}-2$.
This proves the lemma.
\end{proof}

\begin{lem}\label{l bound to finish}
Let $\xi=\{(1,0),(0,1),\gamma_1,\ldots,\gamma_r\}$. Then after at most $u_0(\xi)$ iterations,
the algorithm stops.
\end{lem}
\begin{proof}
This is a direct application of the proof of proposition \ref{p first case}.
\end{proof}

\begin{rem}
Analogous results for the two previous lemmas for the cases $n_s<n_{s-1}<\cdots<n_1<0$, or $(-1,0)$ instead of $(1,0)$, can be reduced
to the previous cases by considering the linear isomorphism $T(x,y)=(-x,y)$, since this isomorphism preserves $L$.
\end{rem}

Now we are ready to give an estimate of how many iterations are needed for the algorithm to stop.
Let $\xi=\{\gamma_1,\ldots,\gamma_r\}\subset\Z^2$ be a set of monomial exponents
of some toric surface. Consider $L(x,y)=ax+by$ with $a$, $b\in\Z$ relatively prime, and such that $L(\xi)\geq0$.
Under these conditions, we can suppose, up to linear isomorphism of determinant 1, that
$\xi\subset\Z\times\Z_{\geq0}$ and $L(x,y)=y$.

\begin{teo}
Let $\xi=\{\gamma_1,\ldots,\gamma_r\}\subset\Z\times\Z_{\geq0}$ be a set of monomial exponents of some
toric surface. Consider $L(x,y)=y$. Then after at most
$$2\cdot u_0(\xi)+2^{u_0(\xi)-1}\cdot v_0(\xi)$$
iterations following $L$, the algorithm stops.
\end{teo}
\begin{proof}
Suppose first that $L(\gamma_i)>0$ for all $i=1,\ldots,r$. If after exactly $u_0(\xi)$ iterations we obtain for
the first time an element of $L-$value 0, say $(n,0)$, then according to lemma \ref{l u_0 iterations},
$\xi_{u_0(\xi)}$ satisfies $0\leq L(\xi_{u_0(\xi)})\leq1$, contains at least two elements of $L-$value 0 such that
their first coordinates are relatively prime, and at least one element of $L-$value 1. In addition,
$v_0(\xi_{u_0(\xi)})\leq2^{u_0(\xi)}\cdot v_0(\xi)$ according to lemma \ref{l bound for n}.
Therefore, if we do not have it already, by lemma \ref{l bound for 1}, after at most $2^{u_0(\xi)-1}\cdot v_0(\xi)$
iterations we will obtain a set $\xi'$ that contains $(1,0)$ (or $(-1,0)$). Since $0\leq L(\xi_{u_0(\xi)})\leq1$,
the set $\xi'$ also satisfies these inequalities. But now having $(1,0)$ (or $(-1,0)$) implies that
the algorithm stops. Summarizing, we needed, at most, $u_0(\xi)+2^{u_0(\xi)-1}\cdot v_0(\xi)$ iterations for 
the algorithm to stop. Since 
$$u_0(\xi)+2^{u_0(\xi)-1}\cdot v_0(\xi)<2\cdot u_0(\xi)+2^{u_0(\xi)-1}\cdot v_0(\xi),$$
the theorem is true in this case.

Suppose now that after $w$ iterations, where $w<u_0(\xi)$, the set $\xi_w$ contains an element $(n,0)$.
Rename $\xi$ as $\xi_0$ and $\xi_w$ as $\xi$. By lemma \ref{l less u_0 iterations}, after $u_1(\xi)$ iterations,
the set $\xi_{u_1(\xi)}$ contains $(1,0)$ (or $(-1,0)$ if $n<0$) or at least two elements of $L-$value 0
such that their first coordinates are relatively prime, and at least one element of $L-$value 1. In addition,
$v_0(\xi_{u_1(\xi)})\leq 2^{u_1(\xi)}\cdot v_0(\xi)\leq2^{u_1(\xi)}\cdot 2^w\cdot v_0(\xi_0),$
according to lemma \ref{l bound for n}. Therefore, after at most $2^{u_1(\xi)+w-1}\cdot v_0(\xi_0)$
iterations we will obtain an element $(1,0)$ (or $(-1,0)$), by lemma \ref{l bound for 1}. Now we are in the situation
of lemma \ref{l bound to finish}. Since $u_0(\xi_k)\leq u_0(\xi_0)$ for any $k\in\N$, then after at most $u_0(\xi_0)$
new iterations the algorithm stops. Summarizing, we needed, at most, 
$$w+u_1(\xi)+2^{u_1(\xi)+w-1}\cdot v_0(\xi_0)+u_0(\xi)$$
iterations for the algorithm to stop. Since $u_1(\xi)\leq u_0(\xi)\leq u_0(\xi_0)-w$, we obtain
$$w+u_1(\xi)+2^{u_1(\xi)+w-1}\cdot v_0(\xi_0)+u_0(\xi)\leq 2\cdot u_0(\xi_0)+2^{u_0(\xi_0)-1}\cdot v_0(\xi_0),$$
and therefore the theorem is also true in this case.

Finally, if $\xi$ already contains some element of $L-$value 0 then we are in the same si\-tuation as in the previous
paragraph without doing the first $w$ iterations. Therefore the result follows similarly. This proves the theorem.
\end{proof}

What about the case where $L(x,y)=ax+by$ with $a$ or $b$ irrational? In all the examples we have computed following such an $L$, 
the algorithm also stops (cf. example \ref{e. ejemplos}, (ii)). However we do not have a proof that this is always the case nor 
we know an example in which the iteration of the algorithm following a linear map $L$ of irrational slope never stops.




\section{\large Local uniformization}

In this section we show that theorem \ref{t. main theorem} implies local uniformization of a toric surface 
for some valuations.
\\
\\
Let $\Gamma$ be an additive abelian totally ordered group. Add to $\Gamma$ an element $+\infty$ such that
$\alpha<+\infty$ for every $\alpha\in\Gamma$ and extend the law on $\Gamma_{\infty}=\Gamma\cup\{+\infty\}$
by $(+\infty)+\alpha=(+\infty)+(+\infty)=+\infty$.

\begin{defi}
Let R be a ring. A valuation of R with values in $\Gamma$ is a mapping $\nu:R\rightarrow\Gamma_{\infty}$ such that:
\begin{itemize}
\item[(i)]$\nu(x\cdot y)=\nu(x)+\nu(y)$ for every $x$, $y\in R$,
\item[(ii)]$\nu(x+y)\geq\min(\nu(x),\nu(y))$ for every $x$, $y\in R$,
\item[(iii)]$\nu(x)=+\infty\Leftrightarrow x=0$.
\end{itemize}
The ring $V=\{x\in R|\nu(x)\geq0\}$ is called the valuation ring associated to $\nu$.
\end{defi}

We will be interested in valuations of the field of rational functions of a toric surface which are trivial
over $\C$. These valuations are classified as follows.

\begin{pro}\label{p. classify groups}
Up to isomorphism, the groups of values $\Gamma$ for valuations of the field of fractions of an algebraic surface over $\C$ are:
\begin{itemize}
\item[(1)]Any subgroup of $\Q$,
\item[(2)]$\Z^2_{lex}$,
\item[(3)]$\Z+\beta\Z$, with $\beta\in\R\setminus\Q$ and $\beta\geq0$.
\end{itemize}
\end{pro}
\begin{proof}
See \cite{Va}, Section 3.2.
\end{proof}

Let $K$ be a field, $\nu$ a valuation of $K$, and $V$ the valuation ring associated to $\nu$.
\begin{defi}
Let $R$ be a subring of $K$. We say that $\nu$ is centered on $R$, or has a center on $R$, if $R\subset V$.
If $X=\mbox{Spec }R$, then we say that $\nu$ is centered on $X$, or has a center on $X$, if has a center on $R$. 
In this case, the center of $\nu$ is the prime ideal of $R$ defined by $R\cap\mathfrak{m}$, where $\mathfrak{m}$ 
is the maximal ideal of $V$.
\end{defi}

\begin{pro}\label{p. center birational}
Let $X$ and $X'$ be two algebraic varieties over $\C$ with the same field of rational functions and let 
$h:X'\rightarrow X$ be a birational and proper morphism. Then any valuation having a center on $X$ has also
a center on $X'$.
\end{pro}
\begin{proof}
See \cite{Va}, Proposition 2.10.
\end{proof}

In the language of schemes, toric surfaces can be characterized as follows:

\begin{lem}
Let $\xi=\{\gamma_1,\ldots,\gamma_r\}\subset\Z^2$ be a set of monomial exponents of some toric surface $X\subset\C^r$.
Consider the morphism of $\C-$algebras,
\begin{align}
\phi:\C[z_1,\ldots,&z_r]\rightarrow\C[x_1,x_2,x_1^{-1},x_2^{-1}]\notag\\
&z_i\mapsto x^{\gamma_i}\notag
\end{align}
Denote by $\C[x^{\xi}]$ the image of $\phi$. Then $X$ is homeomorphic to the set of closed points of $\mbox{Spec}(\C[x^{\xi}])$.
\end{lem}
\begin{proof}
See \cite{CLS}, Chapter 1, Section 1.
\end{proof}

We are now ready to prove the following theorem.

\begin{teo}
Let $\F$ be the field of rational functions of a toric surface. Let $\nu:\F\rightarrow\Gamma$ be any
valuation centered on the toric surface and such that $\nu(x_1)\neq\lambda\nu(x_2)$ for all $\lambda\in\R\setminus\Q$. 
Then a finite iteration of Nash modification gives local uniformization along $\nu$, i. e., the center of the valuation
after those iterations is non-singular.
\end{teo}
\begin{proof}
According to the hypothesis on $\nu$, its possible groups of values are those of (1) and (2) of proposition
\ref{p. classify groups}. Let $\xi=\{\gamma_1,\ldots,\gamma_r\}\subset\Z^2$ be a set of monomial exponents of
the toric variety.
\begin{itemize}
\item[(i)]Consider any valuation $\nu:\F\rightarrow\Q$ centered on $X$ and such that $\nu(x_1)=a$,
$\nu(x_2)=b$. Let $L(t_1,t_2)=at_1+bt_2$. Then $\nu(x^{\gamma_i})=L(\gamma_i)$, and since 
$\nu$ is centered on $X$, we have $L(\xi)\geq0$. After applying Nash modification to $X$, we look
at the affine charts containing the center of $\nu$ (such charts exist according to proposition
\ref{p. center birational}). Suppose that $X'$ is one of these charts. Then we assert that $X'=X_{\xi_{i_0,j_0}}$, 
where the couple $(i_0,j_0)$ is one of the possible choices of $L$. Indeed, the affine charts of the Nash modification of $X$
are of the form $X'=X_{\xi_{i,j}}$ for some $i$, $j$ such that $(0,0)\notin\Co(\xi_{i,j})$. Since $\nu$ is centered on $X'$ 
we have $0\leq\nu(x^{\gamma_k-\gamma_i})=L(\gamma_k-\gamma_i)$ and $0\leq\nu(x^{\gamma_k-\gamma_j})=L(\gamma_k-\gamma_j)$ 
whenever $\gamma_k-\gamma_i$ or $\gamma_k-\gamma_j$ belong to $\xi_{i,j}$ according to (A2) of the algorithm. Assume 
that $L(\gamma_i)\leq L(\gamma_j)$. Then $\gamma_i$, $\gamma_j$ are two elements of $\xi$ such that 
$L(\gamma_i)\leq L(\gamma_k)$ for all $k$, $L(\gamma_j)\leq L(\gamma_k)$ for all $k$ such that 
$\det(\gamma_i\mbox{ }\gamma_k)\neq0$, and also such that $(0,0)\notin\Co(\xi_{i,j})$. This means that $\{\gamma_i,\gamma_j\}$
is one of the possible choices of $L$.
\item[(ii)]Now consider any valuation $\nu:\F\rightarrow\Z^2_{lex}$ centered on $X$ and such that $\nu(x_1)=(a,c)$,
$\nu(x_2)=(b,d)$ with $(a,b)\neq q(c,d)$ for all $q\in\Q$. Let $L(t_1,t_2)=at_1+bt_2$ and $T(t_1,t_2)=ct_1+dt_2$.
As before, $(0,0)\leq\nu(x^{\gamma_i})=(L(\gamma_i),T(\gamma_i))$. In particular, $0\leq L(\xi)$. Arguing as 
in (i), we see that if $X'=X_{\xi_{i,j}}$ is an affine chart of the Nash modification of $X$ in which $\nu$ is centered,
then $\{\gamma_i,\gamma_j\}$ is a possible choice of $L$.
\end{itemize}
Now, by theorem \ref{t. main theorem}, the branches determined by $L$ in the iteration of Nash modification are finite and 
they end in a non-singular surface. In particular, the centers of the valuations considered in (i) and (ii) after these iterations 
are non-singular, that is, this process gives local uniformization along $\nu$.
\end{proof}




\end{document}